\documentclass[a4paper]{article}

\usepackage{amsmath, amsthm, amssymb, mathrsfs}
\usepackage{graphicx, tikz, subfigure, booktabs}
\usepackage{hyperref}
\usepackage{epstopdf}
\usepackage{bm}
\usepackage{makecell}
\usepackage{multirow}

\newtheorem{theorem}{Theorem}[section]
\newtheorem{definition}{Definition}[section]
\newtheorem{problem}{Problem}[section]

\newtheorem{example}{Example}[section]

\begin{document}

\title{Reconstruction of acoustic sources from multi-frequency phaseless far-field data}

\author{
Deyue Zhang\thanks{School of Mathematics, Jilin University, Changchun, China. {\it dyzhang@jlu.edu.cn}},
Yukun Guo\thanks{School of Mathematics, Harbin Institute of Technology, Harbin, China. {\it ykguo@hit.edu.cn} (Corresponding author)},
Fenglin Sun\thanks{School of Mathematics, Jilin University, Changchun, China. {\it sunfl18@jlu.edu.cn}}\ \ and
Xianchao Wang\thanks{School of Mathematics, Harbin Institute of Technology, Harbin, China. {\it xcwang90@gmail.com}}
}
\date{}

\maketitle

\begin{abstract}
We consider the inverse source problem of determining an acoustic source from multi-frequency phaseless far-field data. By supplementing some reference point sources to the inverse source model, we develop a novel strategy for recovering the phase information of far-field data. This reference source technique leads to an easy-to-implement phase retrieval formula. Mathematically, the stability of the phase retrieval approach is rigorously justified. Then we employ the Fourier method to deal with the multi-frequency inverse source problem with recovered phase information. Finally, some two and three dimensional numerical results are presented to demonstrate the viability and effectiveness of the proposed method.
\end{abstract}



\section{Introduction}

The problems of locating or imaging the source excitations using wave propagation make enormous demands on a wide range of realistic engineering applications. Typical scenarios of such inverse source problems include acoustic tomography \cite{AZML07, LU15, SU09} and medical imaging \cite{AM06, ABF02, Arr99, DLU19, FKM04}. In the last decades, great efforts had been devoted to the numerical methods for the inverse source problem of determining an acoustic source. We refer interested readers to \cite{AHLS17, BLLT15, BLRX15, WMGL18, WGZL17, ZG15, ZGLL19} for the sampling method, the continuation method, the eigenfunction expansion method and the Fourier method for recovering the static sources and \cite{WGLL17, WGLL19} for the investigations on imaging the moving point sources.

For the inverse source problems in the frequency domain, the intensity and phase information of the complex-valued radiated data cannot always been measured easily. In fact, in a variety of practical applications, only the modulus or intensity information of the data is available. Therefore, the phaseless inverse problems deserve mathematical and numerical studies. Bao, Lin and Triki \cite{BLT11} introduced the continuation methods for an inverse scattering from phaseless measurements, which is able to capture both the macro structures and small scales of the source term. Zhang et al \cite{ZGLL18} developed a reference point source strategy to recovered the radiated near fields via adding extra artificial sources to the inverse source system. Several numerical methods for the inverse scattering from phaseless data can be found in \cite{AHN19, DZG19, IK11, Kli17, ZGLL18}.

Motivated by \cite{ZGLL18}, we extend the reference point source technique to the case of reconstructing the source function from phaseless far-field data. The incorporation of reference point sources leads to a concise phase retrieval formula and then the inverse source problem is solved by the Fourier method with multi-frequency phased data. The most significant advantage of our method is the computational simplicity since it does not rely on any regularization or iteration process. For similar investigations on phaseless inverse scattering problem, we refer to \cite{JLZ18, JLZ19}.

The rest of this paper is organized as follows. Section \ref{sec:problem_setup} introduces the reference source technique and establish a phase retrieval formula for the far-field data. The results of stability are given in section \ref{sec:stability}. Section \ref{sec:Fourier_method} severs as a brief review of the Fourier method with phased far-field data. In section \ref{sec:numerics}, we present several two- and three-dimensional numerical examples to show the effectiveness of our method.

\section{Reference source technique and the phase retrieval formula}\label{sec:problem_setup}

We first introduce the model problem in $\mathbb{R}^m, m=2, 3$. Let $S \in L^2(\mathbb{R}^m)$ be a source to the homogeneous Helmholtz equation
 \begin{equation}\label{eq:Helmholtz_eqn}
 \Delta u+k^2 u=S \quad \text{in}\ \mathbb{R}^m,
 \end{equation}
where $k>0$ is the wavenumber. We assume that $S$ is independent of $k$ and $\mathrm{supp} S \subset \subset D$, where $D$ is a rectangle in two dimensions or a cuboid in three dimensions centered at the origin. The field $u$ satisfies the Sommerfeld radiation condition
\begin{equation}\label{eq:radiation}
\lim_{r\to \infty}  r^{\frac{m-1}{2}}\left(\frac{\partial u}{\partial r}-\mathrm{i}ku \right)=0, \quad r=|x|.
\end{equation}
Then, the solution to \eqref{eq:Helmholtz_eqn}-\eqref{eq:radiation} can be represented as
\begin{equation*}
u(x, k)=-\int_{D} \Phi_k (x, y)S(y)\,\mathrm{d}y,
\end{equation*}
where
\begin{equation*}
\Phi_k (x, y)=
\begin{cases}
\dfrac{\mathrm{i}}{4} H_0^{(1)}(k|x-y|), & m=2, \vspace{2mm} \\
\dfrac{\mathrm{e}^{\mathrm{i}k|x-y|}}{4\pi|x-y|}, & m=3,
\end{cases}
\end{equation*}
is the fundamental solution to the Helmholtz equation, and $H_0^{(1)}$ denotes the zeroth-order Hankel function of the first kind. Further, $u(x,k)$ admits the asymptotic behavior of the form
\begin{equation}\label{eq:far-field}
u(x,k)=\frac{\mathrm{e}^{\mathrm{i}k|x|}}{|x|^{\frac{m-1}{2}}} \left\{u_\infty(\hat x, k)+\mathcal{O}\left(\frac{1}{|x|}\right)\right\}, \quad |x|\rightarrow \infty,
\end{equation}
uniformly in all directions $\hat x=x/|x|$. In \eqref{eq:far-field}, the complex function $u_\infty$ is known as the far-field pattern and given by
\begin{equation*}
u_\infty(\hat x, k)=-\gamma_m \int_{D} S(y)\mathrm{e}^{-\mathrm{i}k \hat x\cdot y}\,\mathrm{d}y,\quad \hat x\in \Omega,
\end{equation*}
where
$$
\gamma_m=
\begin{cases}
 \dfrac{\mathrm{e}^{\mathrm{i}\pi/4}}{\sqrt{8\pi k}},  & m=2, \vspace{2mm} \\
 \dfrac{1}{4\pi},  & m=3,
\end{cases}
$$
and $\Omega$ is the unit sphere in $\mathbb{R}^3$. The inverse problem considered in this paper can be stated as follows.

\begin{problem}\label{pro:inverse_source_phaseless}
Let $N \in \mathbb{N}$ and $\mathbb{K}_N$ be an admissible set consisting of a finite number of wavenumbers. Then the inverse source problem is to reconstruct an approximation of the source $S(x)$ from the multi-frequency phaseless far-field data $\{|u_\infty(\hat x, k)|:\hat x\in \Omega, k\in \mathbb{K}_N\}$.
\end{problem}

It is well known that the source $S$ cannot be uniquely determined from the phaseless data.
On this account, we will next resort to the reference point source technique to tackle the non-uniqueness issue of the problem.

In this work, we will decompose Problem \ref{pro:inverse_source_phaseless} into two subproblems to overcome the difficulty of non-uniqueness: (i) recovering the phase information of the far-field data and (ii) reconstructing the source with phased data. Motivated by the reference source technique proposed in \cite{ZGLL18}, we added two point sources into the inverse source system and derive a linear system of equations for the far-field data. The locations and intensities of the reference point sources should be suitably chosen such that the linear system admits a unique solution, which will be discussed in detail in the next section. Based on the retrieved far-field data, we would be able to reconstruct the source term by the standard Fourier method developed in \cite{WGZL17}.

Without loss of generality, let $D=(-a/2, a/2)^m$ and $\hat{x}$ denote the observation direction.
Take two points $z_j:=\alpha_j\hat{x}$, where $j=1,2$, $\alpha_j\in\mathbb{R}, m=2,3$. Let $\delta_j$ be the Dirac distributions at points $z_j$ and $c_j$  be the scaling factors
\begin{equation*}
c_j:=\frac{\|u_{\infty}(\cdot,k)\|_{\infty}}{\|\Phi_{k,{\infty}}(\cdot,z_j)\|_{\infty}} ,\quad j=1,2,
\end{equation*}
where $\|\cdot\|_{\infty}=\|\cdot\|_{L^{\infty}(\Omega)}$ and $\Phi_{k,{\infty}}(\hat {x},z_j)=\gamma_m \mathrm{e}^{-\mathrm{i}k \hat {x} \cdot z_j}$ denotes the far-field pattern of fundamental solution at points $z_j$. Thus, the scaling factors $c_j$ take the following form
\begin{equation}\label{eq:scaling_factor}
c_j=\frac{\|u_{\infty}(\cdot,k)\|_{\infty}}{|\gamma_m|} , \quad j=1,2.
\end{equation}

\begin{figure}\label{fig:reference_source}
\centering
\includegraphics[width=0.6\textwidth]{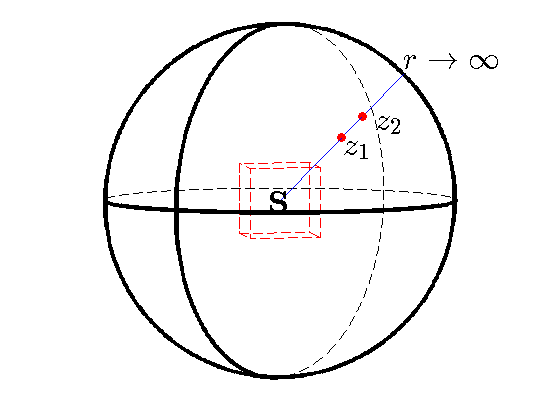}
   \caption{An illustration of the reference source technique, where $z_1$ and $z_2$ denote two reference point sources.}\label{fig:geometry}
\end{figure}

We refer to Figure \ref{fig:reference_source} for an illustration of geometrical configuration for the reference source technique. Now $ \phi_j(x, k):=-c_j\Phi_k(x, z_j)$ satisfies the following inhomogeneous Helmholtz equation
\begin{equation*}
\Delta \phi_j+k^2 \phi_j=c_j \delta_j \quad \text{in}\ \mathbb{R}^m.
\end{equation*}
By the linearity of wave equations, $v_j=u+\phi_j (j=1,2)$ is the unique solution to the problem
\begin{equation*}
\begin{cases}
\Delta v_j+k^2 v_j=S+c_j \delta_j \quad \mathrm{in}\ \mathbb {R}^m,\\
\lim\limits_{r\to \infty}  r^{\frac{m-1}{2}}\left(\dfrac{\partial v_j}{\partial r}-\mathrm{i}k v_j \right)=0, \quad r=|x|.
\end{cases}
\end{equation*}

We denote by $v_{j,\infty}$ and $\phi_{j, \infty}$ the far-field pattern corresponding to $v_j$ and $\phi_j$, respectively. The linearity again leads to
$$
v_{j,\infty}=u_{\infty}+\phi_{j,\infty}=u_{\infty}-c_j\Phi_{k,\infty}=u_{\infty}-c_j\gamma_m \mathrm{e}^{-\mathrm{i}k\hat {x}\cdot z_j}, \quad j=1,2.
$$

The phase retrieval technique in this paper depends on the frequency and is combined with the Fourier method \cite{WGZL17}, so we need to recall the following definition of admissible frequencies \cite{WGZL17}.
\begin{definition}[Admissible wavenumbers]
	Let $N \in \mathbb{N}_+$ and  $k^* \in \mathbb{R}_+$ be a sufficiently small wavenumber. Then the admissible set of wavenumbers is given by
\begin{equation*}
\mathbb{K}_N:=\left\{\frac{2\pi}{a} |\bm l|: \bm l \in \mathbb{Z}^m , 1 \leq |\bm{l}|_\infty\leq N\right\}\cup\{k^* \}.
\end{equation*}
where $m=2,3$, $k^*=2\pi\lambda/a$ and $\lambda$ is a sufficiently small positive constant.
\end{definition}

We now introduce the phase retrieval problem under consideration.
\begin{problem}[Phase retrieval]
Let $v_{j, \infty}$ be the far-field pattern corresponding to the radiated field $v_j, j=1,2$. Given $N\in \mathbb{N}_+$ and the phaseless far-field data
\begin{align*}
& \{|u_{\infty}(\hat x,k)|:\hat x\in \Omega, k\in \mathbb{K}_N\},\\
& \{|v_{j,{\infty}}(\hat x,k)|:\hat x\in \Omega, k\in \mathbb{K}_N\},\quad j=1,2.
\end{align*}
recover the phased data $\{u_{\infty}(\hat x, k):\hat x\in \Omega, k\in \mathbb{K}_N\}$.
\end{problem}

For simplicity, we denote $u_\infty=u_{\infty}(\hat x,k)$, $v_{j,\infty}=v_{j,{\infty}}(\hat x,k)$. Let $t_{k,j}:=k\hat {x}\cdot z_j$, then one can easily derive
\begin{align*}
	\Re v_{j,\infty}= & \Re u_{\infty}-c_j\Re\gamma_m\cos t_{k,j}-c_j\Im\gamma_m\sin t_{k,j}, \quad j=1,2, \\
	\Im v_{j,\infty}= & \Im u_{\infty}+c_j\Re\gamma_m\sin t_{k,j}-c_j\Im\gamma_m\cos t_{k,j}, \quad j=1,2.
\end{align*}
Consequently, we obtain the following equations:
\begin{align}
|u_{\infty}|^2=&(\Re u_{\infty})^2+(\Im u_{\infty})^2, \label{eq:modulus_u} \\
|v_{j,{\infty}}|^2=&(\Re u_{\infty}-c_j\Re\gamma_m\cos t_{k,j}-c_j\Im\gamma_m\sin t_{k,j})^2\notag \\
&+(\Im u_{\infty}+c_j\Re\gamma_m\sin t_{k,j}-c_j\Im\gamma_m\cos t_{k,j})^2, \quad j=1,2. \label{eq:modulus_v}
\end{align}
Further, by subtracting \eqref{eq:modulus_u} from \eqref{eq:modulus_v}, we get
\begin{align}
f_j=&(-\Im\gamma_m \sin t_{k, j}-\Re\gamma_m \cos t_{k, j})\Re u_{\infty} \notag \\
&+(\Re\gamma_m\sin t_{k, j}-\Im\gamma_m\cos t_{k, j})\Im u_{\infty}, \label{eq:eqn_system}
\quad j=1,2,
\end{align}
where
\begin{equation*}
f_j=\frac{1}{2c_j}(|v_{j,\infty}|^2-|u_\infty|^2-c_j^2|\gamma_m|^2) , \quad j=1,2.
\end{equation*}

Thus, we can derive the phase retrieval formula as follows:
\begin{equation}\label{eq:retrieval_formula}
\Re u_{\infty}=\frac{\det A^R}{\det A},\quad
\Im u_{\infty}=\frac{\det A^I}{\det A},
\end{equation}
where the function matrices $A^R$, $A^I$ and $A$, are defined as follows
\begin{align*}
	A^R = &
	\begin{pmatrix}
		f_1 & \Re\gamma_m\sin t_{k,1}-\Im\gamma_m\cos t_{k,1}  \\
		f_2 & \Re\gamma_m\sin t_{k,2}-\Im\gamma_m\cos t_{k,2}  \\
	\end{pmatrix}, \\
	A^I = &
	\begin{pmatrix}
		-\Im\gamma_m\sin t_{k,1}-\Re\gamma_m\cos t_{k,1} & f_1  \\
		-\Im\gamma_m\sin t_{k,2}-\Re\gamma_m\cos t_{k,2} & f_2  \\
	\end{pmatrix}, \\
	A = &
	\begin{pmatrix}
		-\Im\gamma_m\sin t_{k,1}-\Re\gamma_m\cos t_{k,1} & \Re\gamma_m\sin t_{k,1}-\Im\gamma_m\cos t_{k,1}  \\
		-\Im\gamma_m\sin t_{k,2}-\Re\gamma_m\cos t_{k,2} & \Re\gamma_m\sin t_{k,2}-\Im\gamma_m\cos t_{k,2}  \\
	\end{pmatrix}.
\end{align*}
Hence, the far-field can be recovered from $u_{\infty}=\Re u_{\infty}+\mathrm{i}\Im u_{\infty}$.

In this paper, we choose the parameters $\alpha_j$ for $j=1,2$ as follows:
\begin{equation}\label{eq:parameters}
k^*=\frac{\pi}{9a}, \quad \alpha_1=\frac{1}{2} \quad \text{and}\quad
\begin{cases}
\alpha_2=\dfrac{1}{2}-\dfrac{\pi}{2k}, & \text{if}\quad k \in  \mathbb{K}_N\backslash\{k^*\}, \vspace{2mm}  \\
\alpha_2=-4,  & \text{if} \quad k=k^*.
\end{cases}
\end{equation}

In the next section, we will show that the choice of the parameters \eqref{eq:parameters} will make the denominators in \eqref{eq:retrieval_formula} nonzero and thus the equation system \eqref{eq:eqn_system} can be uniquely solvable.

In practice, the collected data should always be polluted by measurement noise. Therefore, we state the algorithm with perturbed data at the end of this section.
\begin{table}[ht]
\centering
\begin{tabular}{cp{.8\textwidth}}
\toprule
\multicolumn{2}{l}{{\bf Algorithm PR:}\quad Phase retrieval with reference point sources} \\
\midrule
{\bf Step 1} & Take the parameters $\alpha_j$ for $j=1,2$ as in \eqref{eq:parameters}; \\
{\bf Step 2} & Measure the noisy phaseless far-field data $\{|u_{\infty}^{\epsilon}(\hat x,k)|:\hat x \in \Omega,  k\in \mathbb{K}_N \}$ and evaluate the scaling factors $c_j$ for $j=1,2$; \\
{\bf Step 3} & Add the reference point sources $c_j\delta_j, j=1,2$, respectively,
 to the inverse source system $S$, and detect the phaseless far-field data $\{|v_{j,\infty}^{\epsilon}(\hat x,k)|:\hat x \in \Omega,  k\in \mathbb{K}_N \}$
 for $j=1,2$; \\
{\bf Step 4} & Recover the far-field pattern $\{u_{\infty}^{\epsilon}(\hat x, k): \hat x \in \Omega,  k\in \mathbb{K}_N \}$ for $j=1,2$ from formula \eqref{eq:retrieval_formula}. \\
\bottomrule
\end{tabular}
\end{table}

\section{Stability of phase retrieval technique}\label{sec:stability}

Because of the similarity of the two cases ($m=2$ and $m=3$), in this section, we only analyze the stability of the phase retrieval method in the two-dimensional case. First, we derive an estimate on $\det A$, which will be crucial in our subsequent stability analysis.
\begin{theorem}\label{thm:lower_bound}
Let $k \in \mathbb{K}_N$, then the following result holds
$$
|\det A|=
\begin{cases}
\dfrac{1}{8\pi k},\quad &  k \in \mathbb{K}_N\backslash\ \{k^*\}, \vspace{2mm} \\
\dfrac{9}{8\pi^2}, \quad &  k=k^*.
\end{cases}
$$
\end{theorem}

\begin{proof}
According to the definition of matrix $A$, $t_{k,j}=k\hat {x} \cdot z_j$ and $z_j=\alpha_j \hat x$ for $j=1,2$, we have
\begin{align*}
\det A = & (-\Im\gamma_m\sin t_{k,1}-\Re\gamma_m\cos t_{k,1})(\Re\gamma_m\sin t_{k,2}-\Im\gamma_m\cos t_{k,2}) \\
&-(\Re\gamma_m\sin t_{k,1}-\Im\gamma_m\cos t_{k,1})(-\Im\gamma_m\sin t_{k,2}-\Re\gamma_m\cos t_{k,2}) \\
=& \sin t_{k,1}\cos t_{k,2}(\Im\gamma_m)^2-\cos t_{k,1}\sin t_{k,2}(\Re\gamma_m)^2 \\
&+\sin t_{k,1}\cos t_{k,2}(\Re\gamma_m)^2-\cos t_{k,1}\sin t_{k,2}(\Im\gamma_m)^2 \\
=& |\gamma_m|^2(\sin t_{k,1}\cos t_{k,2}-\cos t_{k,1}\sin t_{k,2}) \\
=& |\gamma_m|^2\sin(t_{k,1}-t_{k,2}) \\
=& |\gamma_m|^2|\sin(k(\alpha_1-\alpha_2))|.
\end{align*}
Thus, by the parameters in \eqref{eq:parameters}, one easily calculates that
\begin{equation*}
|\det A|=\dfrac{1}{8\pi k}.
\end{equation*}

In the case $k=k^{*}$, one immediately has $k=\dfrac{\pi}{9}$, and
\begin{equation*}
|\det A|=\dfrac{9}{8\pi^2}.
\end{equation*}
This completes the proof.
\end {proof}

Now we move on to the stability of the phase retrieval formula. Given a fixed $k$ and $j$, we consider the perturbed equation system for the unknowns $\Re u_\infty^{\epsilon}$ and $\Im u_\infty^{\epsilon}$:
\begin{equation}\label{eq:perturbed_system}
(-\Im\gamma_m\sin t_{k, j}-\Re\gamma_m \cos t_{k, j})\Re u_{\infty}^{\epsilon}
+(\Re\gamma_m\sin t_{k, j}-\Im\gamma_m\cos t_{k, j})\Im u_{\infty}^{\epsilon}= f_j^{\epsilon},
\end{equation}
where
\begin{equation}\label{eq:perturbed_f_c}
f_j^{\epsilon}=\dfrac{1}
{2c_j^{\epsilon}}(|v_{j,\infty}^{\epsilon}|^2-|u_{\infty}^{\epsilon}|^2-(c_j^{\epsilon})^2|\gamma_m|^2),\quad c_j^{\epsilon}=\dfrac{\| u_{\infty}^{\epsilon}(\cdot,k)\|_\infty}{|\gamma_m|},
\quad j=1,2. 
\end{equation}
Here $u_{\infty}^{\epsilon}$ and $v_{j,\infty}^{\epsilon}$ are measured noisy data satisfying
\begin{equation}\label{eq:noisy_data}
\| |u_{\infty}^{\epsilon}|-|u_{\infty}| \|_{\infty} \leq \epsilon\|u_{\infty}\|_{\infty},\quad \| |v_{j,\infty}^{\epsilon}|-|\hat{v}_{j,\infty}| \|_{\infty} \leq \epsilon\|\hat{v}_{j,\infty}\|_{\infty},
\end{equation}
where $0<\epsilon<1$, $j=1,2$, and
$\hat{v}_{j,\infty}=u_{\infty}(\cdot, k)-c_j^{\epsilon}\Phi_{k,\infty}(\cdot,z_j)$. It can be easily seen that the solutions to the perturbed equations can also be derived by \eqref{eq:eqn_system} with $f_j^\epsilon$ in place of $f_j(j=1,2)$. The main stability result is the following.
\begin{theorem}
Under the above assumptions, we have the estimate
\begin{equation}\label{eq:estimate_u}
\|u_{\infty}^{\epsilon}-u_{\infty} \|_{\infty}\leq C_\epsilon \|u_{\infty}\|_{\infty},
\end{equation}
where
$$
C_\epsilon = 2\epsilon\dfrac{(\epsilon+3)(\epsilon+2)^2+6}{1-\epsilon}.
$$
\end{theorem}

\begin{proof}
By $v_{j,\infty}(\cdot,k)=u_\infty(\cdot,k)-c_j\Phi_{k,\infty}(\cdot,z_j)$ and
$\|\Phi_{k,\infty}(\cdot,z_j) \|_\infty = |\gamma_m|$, \eqref{eq:scaling_factor}, \eqref{eq:perturbed_f_c} and \eqref{eq:noisy_data}, we deduce that
\begin{equation}\label{eq:estimate_c}
|c_j^{\epsilon}-c_j|\leq \dfrac{\epsilon \|u_{\infty}(\cdot,k)\|_{\infty}}{|\gamma_m|}=\epsilon c_j,
\end{equation}
and
\begin{align}
\| |v_{j,\infty}^{\epsilon}|-|v_{j,\infty}| \|_{\infty}
& \leq \| |v_{j,\infty}^{\epsilon}|-|\hat {v}_{j,\infty}| \|_{\infty}+\| |\hat {v}_{j,\infty}|-|v_{j,\infty}| \|_{\infty} \notag \\
& \leq \epsilon\|\hat {v}_{j,\infty}\|_{\infty}+\| (c_j^{\epsilon}-c_j)\Phi_{k,\infty}(\cdot, z_j) \|_{\infty} \notag \\
& \leq \epsilon\|\hat {v}_{j,\infty}\|_{\infty}+\epsilon c_j |\gamma_m| \notag \\
& \leq \epsilon \|u_{\infty}(\cdot,k) \|_{\infty}+\epsilon(c_j+c_j^{\epsilon})|\gamma_m| \notag \\
& =\epsilon(2\|u_{\infty}(\cdot,k) \|_{\infty}+\|u_{\infty}^{\epsilon}(\cdot,k) \|_{\infty})\notag \\
& \leq \epsilon(\epsilon+3) \|u_{\infty}(\cdot,k) \|_{\infty}. \label{eq:estimate_v}
\end{align}
Further, by using \eqref{eq:perturbed_f_c}, \eqref{eq:estimate_c} and \eqref{eq:estimate_v}, we derive
\begin{align*}
\left|  \frac{|v_{j,\infty}^{\epsilon}|^2}{c_j^{\epsilon}}-\frac{|v_{j,\infty}|^2}{c_j} \right|
 \leq & \frac{1}{c_j^{\epsilon}}\left| |v_{j,\infty}^{\epsilon}|^2-|v_{j,\infty}|^2 \right|+|v_{j,\infty}|^2\left|\frac{1}{c_j^{\epsilon}}-\frac{1}{c_j}\right|\\
 \leq &\frac{\epsilon(\epsilon+3)}{c_j^{\epsilon}} \|u_{\infty}(\cdot, k) \|_{\infty} \left(|v_{j,\infty}^{\epsilon}|+|v_{j,\infty}|\right)\\
&+\dfrac{|c_j^{\epsilon}-c_j|}{c_j^{\epsilon}c_j}|v_{j,\infty}|^2\\
 \leq &\frac{\epsilon(\epsilon+3)(\epsilon+2)^2}{c_j^{\epsilon}}\|u_{\infty}(\cdot, k) \|_{\infty}^2+\frac{\epsilon}{c_j^{\epsilon}}\|v_{j,\infty}\|_{\infty}^2\\
 \leq &\frac{\epsilon((\epsilon+3)(\epsilon+2)^2+2)}{(1-\epsilon)c_j}\|u_{\infty}(\cdot, k) \|_{\infty}^2\\
 =& \frac{\epsilon((\epsilon+3)(\epsilon+2)^2+2)}{1-\epsilon}\|u_{\infty}(\cdot, k) \|_{\infty} |\gamma_m|.
\end{align*}

Analogously, we can get
\begin{align*}
\left|  \frac{|u_{\infty}^{\epsilon}|^2}{c_j^{\epsilon}}-\frac{|u_{\infty}|^2}{c_j} \right|
\leq & \frac{\epsilon(\epsilon+3)}{1-\epsilon}\|u_{\infty}(\cdot, k) \|_{\infty}|\gamma_m|,\\
|c_j^{\epsilon}-c_j||\gamma_m|^2 \leq & \epsilon\|u_{\infty}(\cdot, k) \|_{\infty}|\gamma_m|.
\end{align*}
The triangle inequality implies that
\begin{equation}\label{eq:estimate_f}
|f_j^{\epsilon}-f_j|\leq \eta_\epsilon\|u_{\infty}(\cdot, k)\|_{\infty}|\gamma_m|, \quad j=1,2, 
\end{equation}
where
$$
\eta_\epsilon:=\epsilon\dfrac{(\epsilon+3)(\epsilon+2)^2+6}{2(1-\epsilon)}.
$$

According to the admissible wavenumbers, the subsequent proof is divided into two cases.

Case (i): $k\in\mathbb{K}_N\backslash\{k^*\}$.

It is clear to see that, for $j=1,2$
\begin{align*}
& |\Re\gamma_m \sin t_{k, j}-\Im\gamma_m \cos t_{k, j}|
\leq |\Re\gamma_m|+|\Im\gamma_m| \leq \sqrt{2} |\gamma_m|, \\
& |\Im\gamma_m \sin t_{k, j}+\Re\gamma_m \cos t_{k, j}|
\leq |\Im\gamma_m|+|\Re\gamma_m| \leq \sqrt{2} |\gamma_m|,
\end{align*}
which, together with \eqref{eq:estimate_f}, imply that
\begin{align}
& |(\Re\gamma_m \sin t_{k,2}-\Im\gamma_m \cos t_{k,2})(f_1^\epsilon-f_1) \notag \\
& -(\Re\gamma_m \sin t_{k,1}-\Im\gamma_m \cos t_{k,1})(f_2^\epsilon-f_2)| \notag \\
\leq & \sqrt{2} |\gamma_m|\eta_\epsilon\|u_{\infty}(\cdot, k) \|_{\infty}|\gamma_m|+
\sqrt{2}|\gamma_m|\eta_\epsilon\|u_{\infty}(\cdot, k) \|_{\infty}|\gamma_m| \notag \\
= & 2\sqrt{2}\eta_\epsilon |\gamma_m|^2\|u_{\infty}(\cdot, k) \|_{\infty}. \label{eq:estimate_i1}
\end{align}
and
\begin{align}
& |(\Im\gamma_m\sin\ t_{k,2}+\Re\gamma_m \cos t_{k,2})(f_1^\epsilon-f_1) \notag \\
& -(\Im\gamma_m\sin t_{k,1}+\Re\gamma_m \cos t_{k,1})(f_2^\epsilon-f_2)| \notag \\
 \leq & 2\sqrt{2}\eta_\epsilon |\gamma_m|^2\|u_{\infty}(\cdot, k)\|_{\infty}. \label{eq:estimate_i2}
\end{align}

Subtracting \eqref{eq:eqn_system} from \eqref{eq:perturbed_system} yields
\begin{align}
f_j^{\epsilon}-f_j = & -(\Im\gamma_m\sin t_{k,j}+\Re\gamma_m \cos t_{k,j})\Re (u_{\infty}^{\epsilon}-u_{\infty}) \notag \\
&+(\Re\gamma_m\sin t_{k,j}-\Im\gamma_m\cos t_{k,j})\Im (u_{\infty}^{\epsilon}-u_{\infty}),\quad j=1,2. \label{eq:subtraction}
\end{align}
Thus,
\begin{equation*}
 \Re(u_{\infty}^{\epsilon}-u_{\infty})=\frac{\det A^{R,\epsilon}}{\det A},\quad
 \Im(u_{\infty}^{\epsilon}-u_{\infty})=\frac{\det A^{I,\epsilon}}{\det A}.
\end{equation*}
where
\begin{align*}
	A^{R,\epsilon} = &
	\begin{pmatrix}
		f_1^{\epsilon}-f_1 & \Re\gamma_m\sin t_{k,1}-\Im\gamma_m\cos t_{k,1}  \\
		f_2^{\epsilon}-f_2 & \Re\gamma_m\sin t_{k,2}-\Im\gamma_m\cos t_{k,2}  \\
	\end{pmatrix}, \\
	A^{I,\epsilon} = &
	\begin{pmatrix}
		-\Im\gamma_m\sin t_{k,1} -\Re\gamma_m\cos t_{k,1} & f_1^{\epsilon}-f_1  \\
		-\Im\gamma_m\sin t_{k,2}-\Re\gamma_m\cos t_{k,2} & f_2^{\epsilon}-f_2  \\
	\end{pmatrix}.
\end{align*}
Then, in terms of Theorem \ref{thm:lower_bound}, together with \eqref{eq:estimate_i1} and  \eqref{eq:estimate_i2}, we derive
\begin{equation*}
\|u_{\infty}^{\epsilon}-u_{\infty} \|_\infty \leq C_\epsilon \|u_{\infty}\|_\infty,
\end{equation*}
where
\begin{equation*}
C_\epsilon =2\epsilon\dfrac{(\epsilon+3)(\epsilon+2)^2+6}{1-\epsilon}.
\end{equation*}

Case (ii): $k=k^*$.

In this condition, the parameters have the following values:
\begin{equation*}
k=\frac{\pi}{9},\quad \alpha_1=\frac{1}{2},\quad \alpha_2=-4,\quad
t_{k,1}=\frac{\pi}{18}, \quad  t_{k,2}=-\frac{4\pi}{9}.
\end{equation*}
Then
\begin{align*}
\Re\gamma_m\sin t_{k,1}-\Im\gamma_m\cos t_{k,1} \approx & 0.1937, \\
\Re\gamma_m\sin t_{k,2}-\Im\gamma_m\cos t_{k,2} \approx & -0.2766, \\
\Im\gamma_m\sin t_{k,1}+\Re\gamma_m\cos t_{k,1} \approx & 0.2766, \\
\Im\gamma_m\sin t_{k,2}+\Re\gamma_m\cos t_{k,2} \approx & 0.1937.
\end{align*}
Furthermore, from \eqref{eq:estimate_f}, we can derive
\begin{align}
& |(\Re\gamma_m \sin t_{k,2}-\Im\gamma_m \cos t_{k,2})(f_1^\epsilon-f_1) \notag \\
&-(\Re\gamma_m \sin t_{k,1}-\Im\gamma_m \cos t_{k,1})(f_2^\epsilon-f_2)| \notag \\
\leq & 0.4703 \eta_\epsilon\|u_{\infty}(\cdot,k) \|_{\infty}|\gamma_m| \notag \\
\leq & 0.1588 \eta_\epsilon\|u_{\infty}(\cdot,k) \|_{\infty}, \label{eq:estimate_ii1}
\end{align}
and
\begin{align}
& |(\Im\gamma_m\sin\ t_{k,2}+\Re\gamma_m \cos t_{k,2})(f_1^\epsilon-f_1) \notag \\
&-(\Im\gamma_m\sin t_{k,1}+\Re\gamma_m \cos t_{k,1})(f_2^\epsilon-f_2)| \notag \\
\leq & 0.4703 \eta_\epsilon\|u_{\infty}(\cdot,k) \|_{\infty}|\gamma_m| \notag \\
\leq & 0.1588 \eta_\epsilon\|u_{\infty}(\cdot,k) \|_{\infty}. \label{eq:estimate_ii2}
\end{align}
Then, by solving the equations \eqref{eq:subtraction} with \eqref{eq:perturbed_system}, \eqref{eq:estimate_ii1} and  \eqref{eq:estimate_ii2}, we have completed the proof.
\end{proof}

\section{Fourier method}\label{sec:Fourier_method}

Using the retrieved far-field data in the last section, the phaseless inverse source problem can be reformulated as the standard inverse source problem with phase information.
\begin{problem}[Multi-frequency ISP with far-field]\label{pro:inverse_source}
Given a finite number of frequencies $\{k\}$ and the corresponding far-field data $\{u_{\infty}(\hat x_k, k ): \hat x_k \in \Omega\}$, where $\hat x_k$ depends on the wavenumber $k$, find the source function $S(x)$.
\end{problem}
In this section, we will describe the Fourier method for solving Problem \ref{pro:inverse_source} in brief. For more details on the Fourier method, please see \cite{ZG15, WGZL17} for the acoustic case and \cite{WMGL18, WSGLL19} for the electromagnetic case. The Fourier method relies on the approximation of target source function $S(x)$ by a Fourier series of the form
\begin{equation}\label{eq:Fourier-expansion}
 S(x)=\sum_{\bm l\in \mathbb{Z}^m} \hat{s}_{\bm l}\, \phi_{\bm l}(x).
\end{equation}
where $\hat{s}_{\bm l}, \bm l\in \mathbb{Z}^m,$ are the Fourier coefficients and
\begin{equation*}
  \phi_{\bm l}(x)=\mathrm{exp}\left(\mathrm{i}\frac{2\pi}{a} \bm l \cdot x\right),\quad \bm l\in \mathbb{Z}^m,
\end{equation*}
are the Fourier basis functions.

Let 
\begin{equation*}
  \bm l_0:=
\begin{cases}
 \left(\dfrac{1}{18}, 0\right),     & m=2, \vspace{2mm} \\
 \left(\dfrac{1}{18}, 0, 0\right), & m=3,
\end{cases}
\end{equation*}
then the admissible wavenumbers are defined by
\begin{equation*}
  k_{\bm l}:=\left\{
\begin{aligned}
 & \frac{2\pi}{a}|\bm l|,    \quad \bm l \in \mathbb{Z}^m\backslash\{\bm 0\}, \\
 &  \frac{\pi}{9a}\lambda,   \quad \ \bm l =\bm 0.
\end{aligned}
\right.
\end{equation*}
Correspondingly, the admissible observation directions are given by
\begin{equation*}
  \hat x_{\bm l}:=\left\{
\begin{aligned}
 & \frac{\bm l}{|\bm l|},    \quad \ \bm l \in \mathbb{Z}^m\backslash\{\bm 0\}, \\
 &  \frac{\bm l_0}{|\bm l_0|}, \quad  \bm l =\bm 0.
\end{aligned}
\right.
\end{equation*}
In addition, following \cite{WGZL17}, the Fourier coefficients are given as follows
\begin{align*}
 &\hat{s}_{\bm l}=-\frac{1}{a^m \gamma_m}u_\infty(\hat{x};k_{\bm l}),\quad \bm l \in \mathbb{Z}^m\backslash\{\bm 0\}, \\
 & \hat{s}_{\bm 0} \approx -\frac{\lambda \pi}{a^m \mathrm{\sin}\lambda \pi}\left\{ \frac{u_\infty(\hat{x};k_{\bm 0 })}{\gamma_m} +  \sum_{1\leq|\bm l|_{\infty} \leq N} \hat{s}_{\bm {l}} \int_{D} \phi_{\bm {l}}(y) \overline{\phi_{\bm {l}_0}(y)} \, \mathrm{d} y \right\}.
\end{align*}

The implementation of the Fourier method is briefly stated in \textbf{Algorithm FM}.
\begin{table}[ht]
\centering
\begin{tabular}{cp{.8\textwidth}}
\toprule
\multicolumn{2}{l}{{\bf Algorithm FM:}\quad Fourier method for recovering the source } \\
\midrule
{\bf Step 1} & Select the parameters $\lambda, N $ and the admissible set $\mathbb{K}_N$; \\
{\bf Step 2} & Measure the noisy multi-frequency far-field data $\{u_{\infty}^{\epsilon}(\hat x_{\bm l}, k_{\bm l}):\hat x_{\bm l} \in \Omega,  k_{\bm l}\in \mathbb{K}_N \}$; \\
{\bf Step 3} & Compute the Fourier coefficients $\hat s_{\bm 0}$ and  $\hat s_{\bm l},\, \bm l \in \mathbb{Z}^m\backslash\{\bm 0\}$, and then a truncated series defined in \eqref{eq:Fourier-expansion} is the reconstruction of $S$. \\
\bottomrule
\end{tabular}
\end{table}

\section{Numerical experiments}\label{sec:numerics}

In this section, several two-and three-dimensional numerical experiments are carried out to demonstrate the feasibility and effectiveness of the proposed method. To avoid the inverse crime, we generate the synthetic far-field data by solving the forward problem via the quadratic finite elements method with a perfectly matched layer together with the Kirchhoff integral formula. The mesh of the finite element solver is successively refined till the relative error of the successive discrete solutions is below $0.1\%$. To test the stability of the method, we also add some uniformly distributed random noises to the phaseless far-field data. The perturbed phaseless data was given by:
\begin{equation*}
u_{\infty}^{\epsilon}:=(1+\epsilon r)|u_\infty|,
\end{equation*}
where $r\in[-1,1]$ is a random number and $\epsilon>0 $ represents the noise level.

To quantitatively evaluate the accuracy of the phase retrieval formula, we also estimate the relative error between the exact far-field data and the retrieval data. The discrete relative $L^2$ errors and $L^\infty$ errors are respectively calculated as follows:
$$
\frac{\left(\sum\limits_{i=1}^M |u_\infty(\hat x,k_i)-u_\infty^\epsilon(\hat x,k_i)|^2\right)^{1/2}}{\left(\sum\limits_{i=1}^M |u_\infty(\hat x,k_i)|^2\right)^{1/2}},\quad
\frac{\max\limits_{\hat x, k_i}\left|u_\infty(\hat x, k_i)-u_\infty^\epsilon(\hat x, k_i)\right|}{\max\limits_{\hat x, k_i}|u_\infty(\hat x, k_i)|},
$$
where $M$ denotes the number of frequencies, and $u_\infty(\hat x,k_i)$ and $u_\infty^\epsilon(\hat x, k_i)$ are the exact and retrieved far fields, respectively.
In what follows,  we first present some numerical validations of the phase retrieval technique, namely, Algorithm PR. Then we are going to image the source function with phased far field data by the Fourier method. Following \cite{WGZL17}, the truncation of the Fourier expansion is chosen as $N=2\lceil \epsilon^{-1/3} \rceil$ where the ceiling function $\lceil X \rceil$ denotes the largest integer that is smaller than $X+1$. Here, we use 1681 frequencies in two dimensions (68921 frequencies in three dimensions) and the same number of $\hat x$ to get the far field data $u_\infty(\hat x, k_i)$. For more implementation details, we refer to \cite{WGZL17}.

\subsection{Examples in two dimensions}

\begin{example}\label{E1}
\rm In the first example, we consider a mountain-shaped smooth source function 
\begin{align*}
   S_1(x_1,x_2)=&1.1\mathrm{exp}(-200((x_1-0.01)^2+(x_2-0.12)^2)) \\
                          &-100(x_2^2-x_1^2) \mathrm{exp}(-90(x_1^2+x_2^2)).
\end{align*}
\end{example}
We begin with the validation of phase retrieval algorithm. Figure \ref{fig.mountain-phase-retrieval} presents the  geometry setup of the phase retrieval technique for four typical wavenumbers. We would like to emphasize that the reference point sources depend on the measured direction $\hat x_{\bm l}$, where $\bm l\in \mathbb{Z}^m$. To show the accuracy of phase retrieval technique in Algorithm PR, we list the exact- and reconstructed phase data in Table \ref{tab:phase}. It is clear that the recovered phase data is point-wise convergent to the exact phase date. Furthermore, to show the global accuracy of the phase retrieval technique, we list the relative $L^2$ and $L^{\infty}$ errors in Table \ref{tab:test1}. It can be seen  that the error decreases as the noise level decreases and the phase retrieval procedure is quite accurate and stable.

  \begin{figure}
   \subfigure[]{\includegraphics[width=0.49\textwidth]{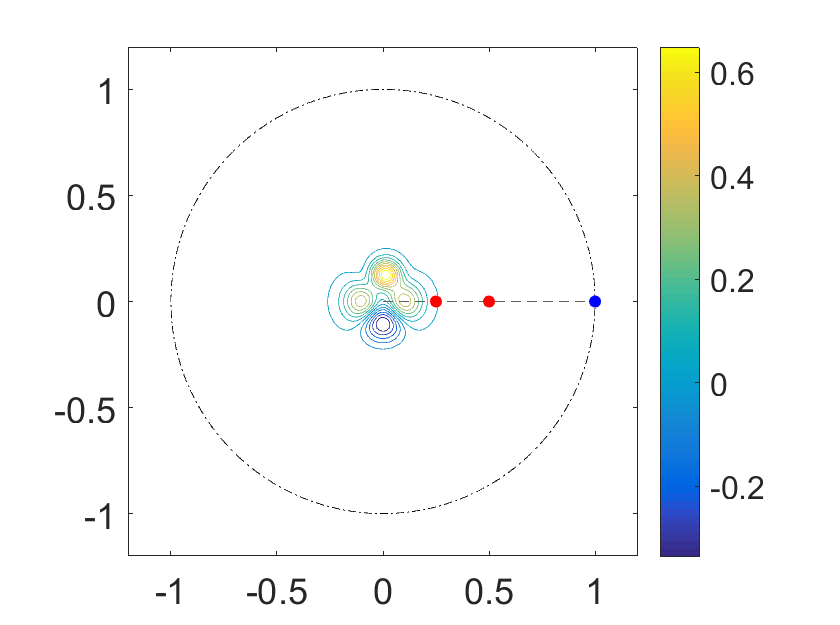}}
   \subfigure[]{\includegraphics[width=0.49\textwidth]{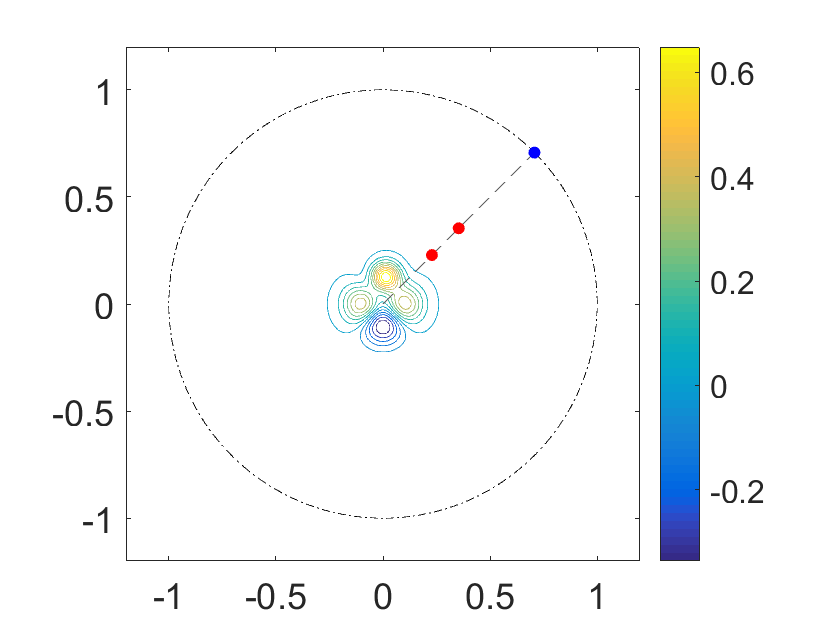}}\\
   \subfigure[]{\includegraphics[width=0.49\textwidth]{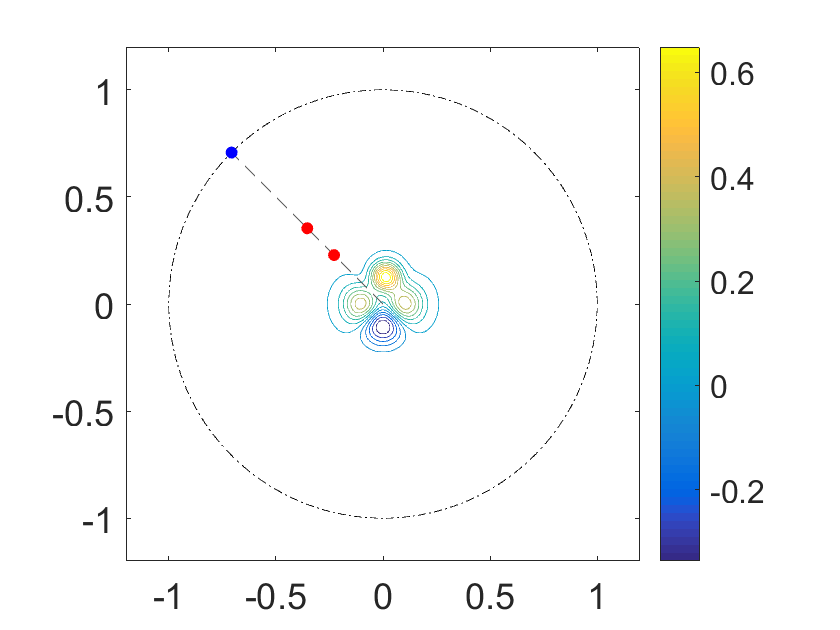}}
   \subfigure[]{\includegraphics[width=0.49\textwidth]{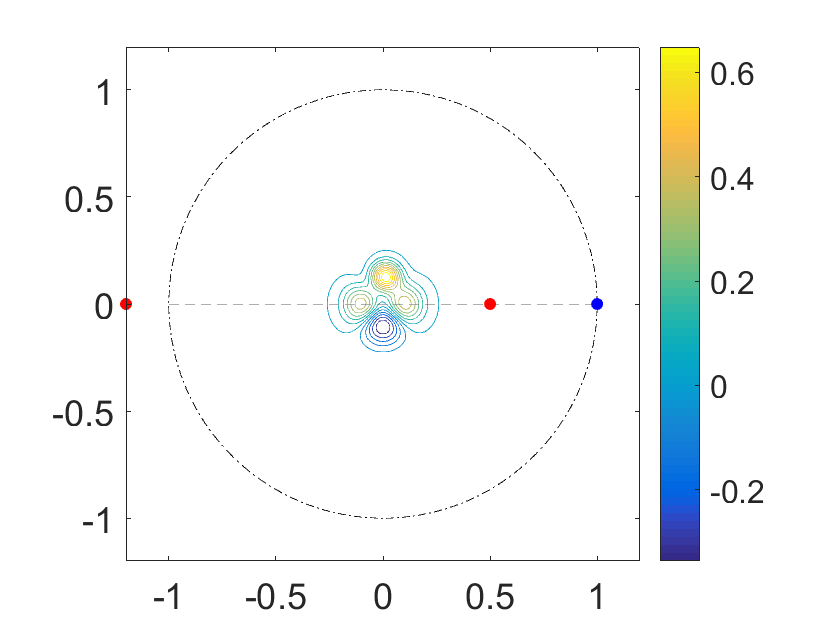}}
   \caption{Geometry illustration of the phase retrieval technique. The red points denote the reference point sources and the blue points denote observation directions. (a) $\bm l=(1,0)$, (b) $\bm l=(1,1)$,  (c) $\bm l=(-1,1)$,  (b) $\bm l=(0,0)$. }\label{fig.mountain-phase-retrieval}
\end{figure}

\begin{table}
\centering
 \caption{Phase retrieval for the mountain-shape source with $\epsilon=1\%$ noise.}\label{tab:phase}
 \begin{tabular}{ccccccc}
 \toprule
 $k$ & $\hat x$  & $\Re u_k^{\infty}(\hat x)$ & $\Im u_k^{\infty}(\hat x)$ 
                   & $\Re u_{k,\epsilon}^{\infty}(\hat x)$ & $ \Im u_{k,\epsilon}^{\infty}(\hat x)$ \\
  \midrule
  \multirow{4}*{$2\pi$}
  & $(1,0)$    & $-7.52\times 10^{-4} $  & $-6.38\times 10^{-4}$ & $-7.46\times 10^{-4} $  & $-6.35\times 10^{-4}$ \\
  & $(0,1)$    & $-1.49\times 10^{-3} $  & $-2.50\times 10^{-4}$ & $-1.51\times 10^{-3} $  & $-2.43\times 10^{-4}$ \\
  & $(-1,0)$   & $-6.38\times 10^{-4} $  & $-7.52\times 10^{-4}$ & $-6.44\times 10^{-4} $  & $-7.48\times 10^{-4}$ \\
  & $(0,-1)$   & $-2.50\times 10^{-4} $  & $-1.49\times 10^{-3}$ & $-2.51\times 10^{-4} $  & $-1.49\times 10^{-3}$ \\
  \midrule
  \multirow{4}*{$2\sqrt{2}\pi$}
  & $(\frac{\sqrt{2}}{2},\frac{\sqrt{2}}{2})$
                 & $-1.03\times 10^{-3} $  & $3.23\times 10^{-5}$
                 & $-1.03\times 10^{-3} $  & $- 1.86\times 10^{-5}$   \\
  & $(-\frac{\sqrt{2}}{2},\,\frac{\sqrt{2}}{2})$
                 & $-1.02\times 10^{-3} $  & $- 9.66\times 10^{-5}$
                 & $-1.04\times 10^{-3} $  & $- 1.33\times 10^{-4}$   \\
  & $(-\frac{\sqrt{2}}{2},\,-\frac{\sqrt{2}}{2})$
               & $3.23\times 10^{-5} $  & $- 1.03\times 10^{-3}$
               & $1.10\times 10^{-5} $  & $- 1.03\times 10^{-3}$   \\
  & $(\frac{\sqrt{2}}{2},\,-\frac{\sqrt{2}}{2})$
               & $-9.66\times 10^{-5} $  & $- 1.02\times 10^{-3}$
               & $-1.09\times 10^{-4} $  & $- 9.77\times 10^{-4}$   \\
   \midrule
  $\pi/9$  &$(1,0)$  & $ -7.39\times 10^{-3}$  & $ -7.37\times 10^{-3}$
                               & $ -7.07\times 10^{-3}$  & $ -7.04\times 10^{-3}$ \\
  \bottomrule
 \end{tabular}
\end{table}

\begin{table}[h]
 \centering
    \setlength{\abovecaptionskip}{0pt}
    \setlength{\belowcaptionskip}{10pt}
 \caption{\label{tab:test1}The relative $L^2$ errors and $L^{\infty}$  errors between $u_\infty (\hat x,k)$ and $u_\infty ^\epsilon(\hat x,k)$ for different noise levels $\epsilon $.}
 \begin{tabular}{lclclcl}
  \toprule
   $S_1(x)$ & $\epsilon=0.1\%$ & $\epsilon=1\%$ & $\epsilon=5\% $ &$\epsilon=10\%$ \\
  \midrule
  $L^2$         & $0.06\%$ & $0.59\%$ & $2.45\%$ &$5.77\%$ \\
  $L^{\infty}$ & $0.09\%$ & $0.91\%$ &$2.85\%$ & $7.74\%$ \\
  \bottomrule
 \end{tabular}
\end{table}

Next, we will use the recovered phased data to reconstruct the source function.
Figure \ref{fig.mountain-phase-retrieval} shows the contour and surface plots of the exact and reconstruction of $S_1$ with noise $\delta=0.1\% $ inside the rectangular domain $[-0.3,0.3]\times[-0.3,0.3]$.

\begin{figure}
	\centering
	\subfigure[]{\includegraphics[width=0.49\textwidth]{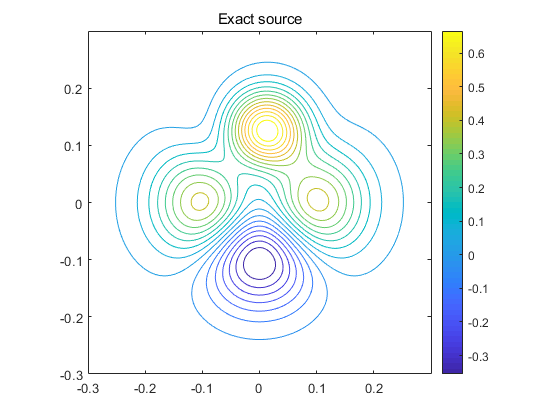}}
	\subfigure[]{\includegraphics[width=0.49\textwidth]{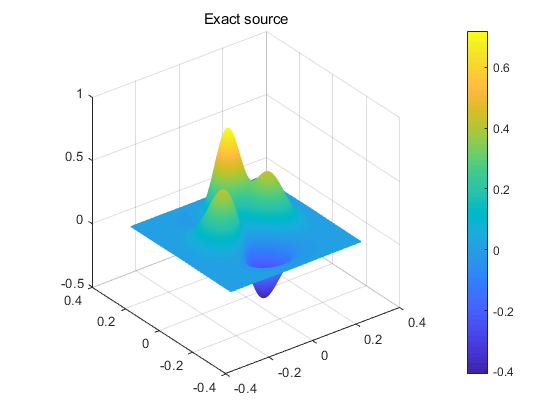}}
        \subfigure[]{\includegraphics[width=0.49\textwidth]{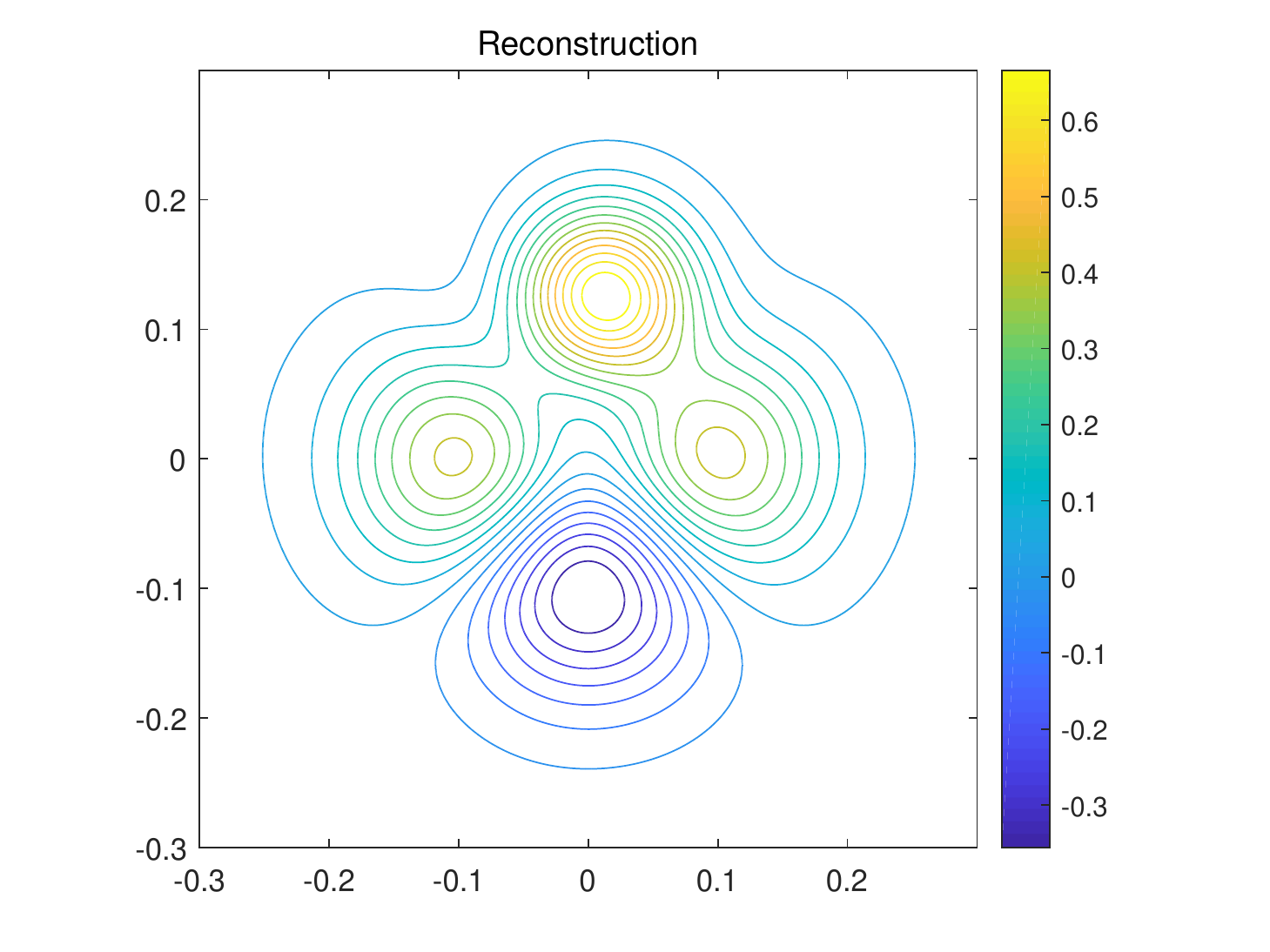}}
	\subfigure[]{\includegraphics[width=0.49\textwidth]{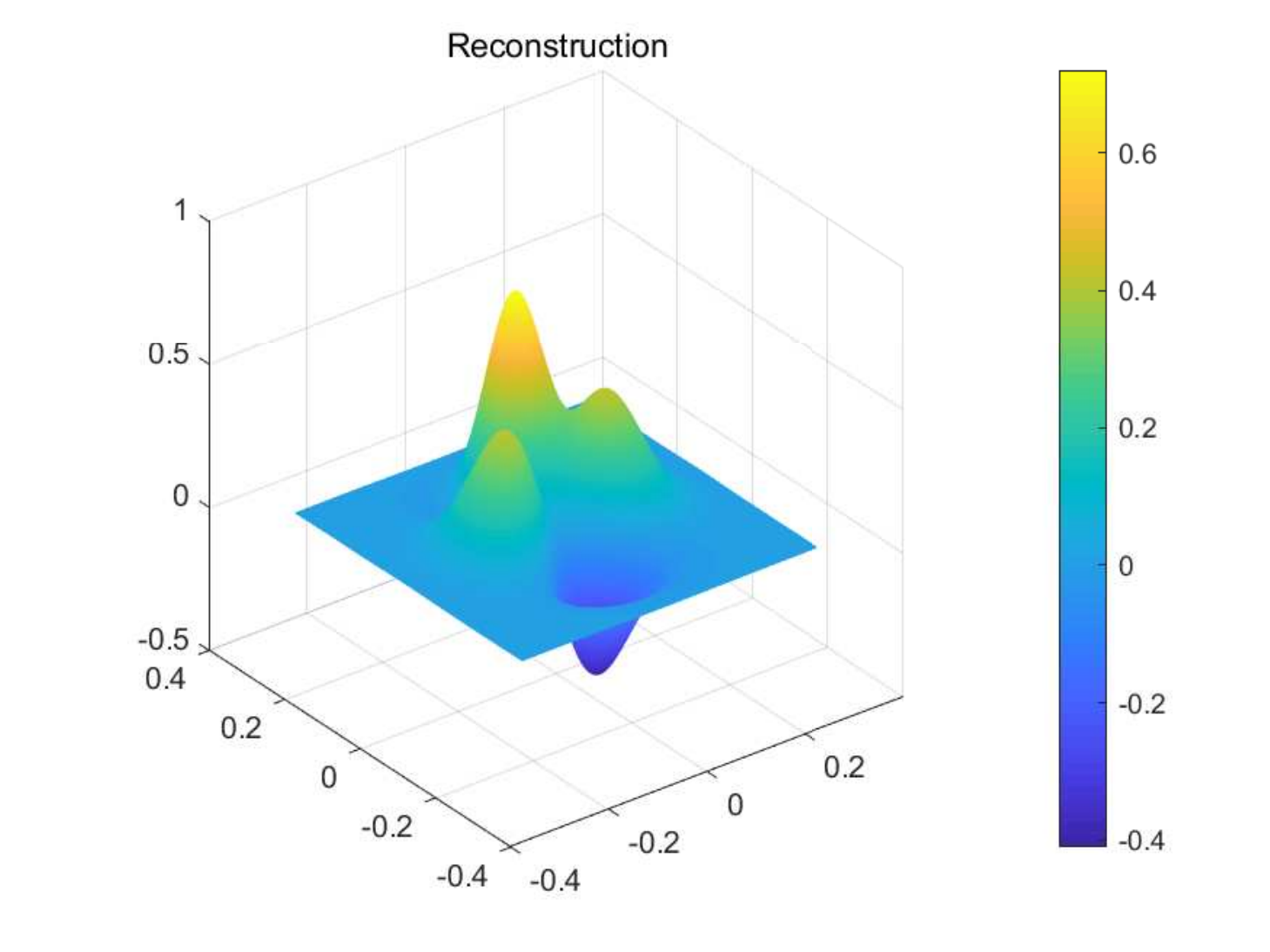}}
	\caption{Top row: The exact source function $S_1$. (a) surface plot (b) contour plot. Bottom row: The reconstruction of source function $S_1$. (c) surface plot (d) contour plot.} \label{fig:Source-mountain}
\end{figure}

\begin{example}\label{E2}
\rm In the second example, we consider a discontinuous source defined by
$$
S_2(x)=
\begin{cases}
1.5, &  x\in D_1  \\
2,    &  x\in D_2 \\
1,    &  x\in D_3 \\
0,    &  \text{elsewhere},
\end{cases}
$$
where $D_1, D_2$ and $D_3$ have respectively the following parametric boundaries
\begin{align*}
 &D_1:\ \frac{2+0.3\cos 3t}{15}(\cos t- 0.1,\ \sin t+0.2 ),\quad t\in[0, 2\pi],\\
 &D_2:\ \frac{0.1+0.08\cos t+0.02\sin 2t}{1+0.7\cos t} (\cos t- 0.2,\ \sin t-0.2 ),\quad t\in[0, 2\pi], \\
 &D_3:\ (0.1\cos t+0.065\cos 2t+0.135, 0.15\sin t-0.2),\quad t\in[0, 2\pi].\\
\end{align*}
\end{example}

Figure \ref{fig:discontinuous} shows the contour plots of the exact- and reconstructed source $S_2$ with different noise levels. One can find that the proposed method has the capability of reconstructing a piecewise constant source function consisting of three disconnected components. Moreover, according to Figure \ref{fig:discontinuous}(c) and (e), we can also see that the relative error occurs mainly on the boundary of the source. This is due to the fact that the the Gibbs phenomena occurs around the discontinuous border when the Fourier method is applied to recover functions with jumps.

\begin{figure}
  \centering
   \subfigure[]{\includegraphics[width=0.32\textwidth]{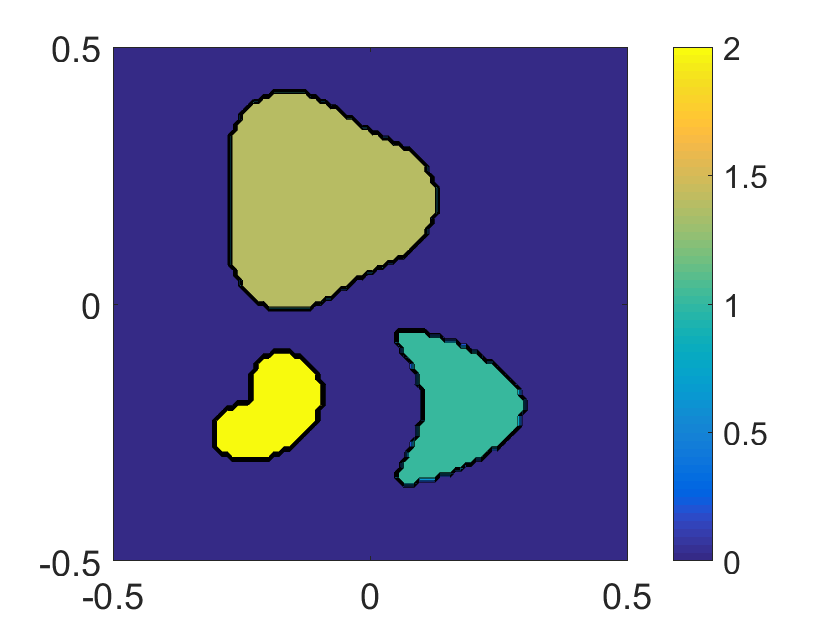}}
   \subfigure[]{\includegraphics[width=0.33\textwidth]{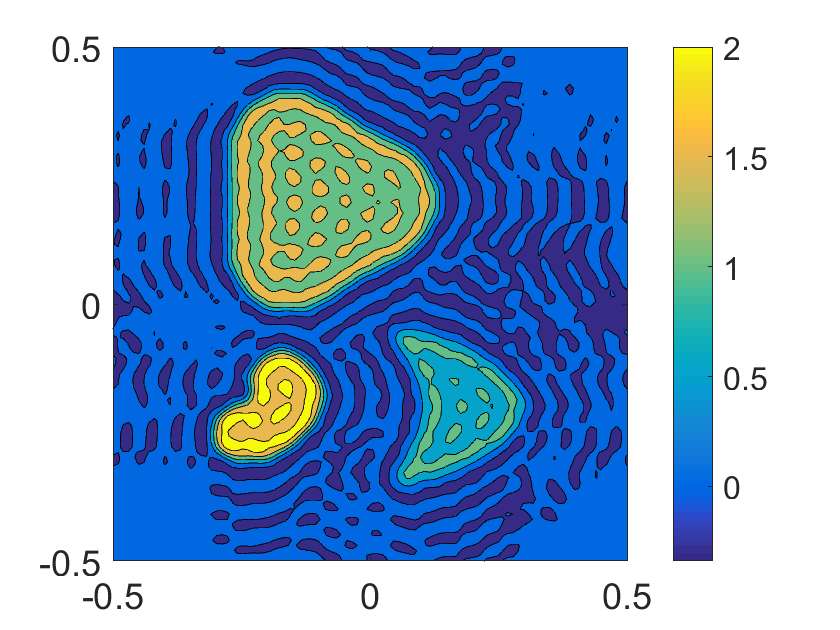}}
   \subfigure[]{\includegraphics[width=0.33\textwidth]{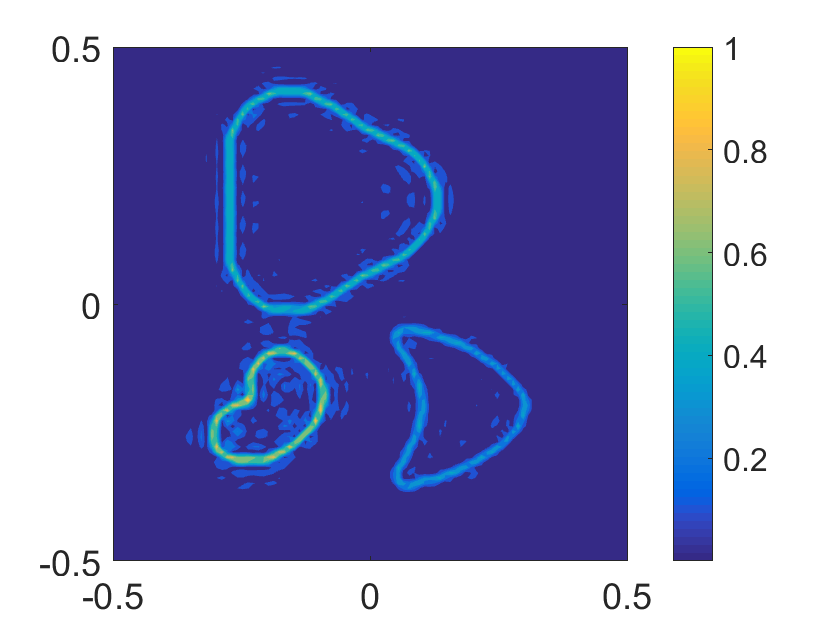}}\\
  \hfill  \subfigure[]{\includegraphics[width=0.33\textwidth]{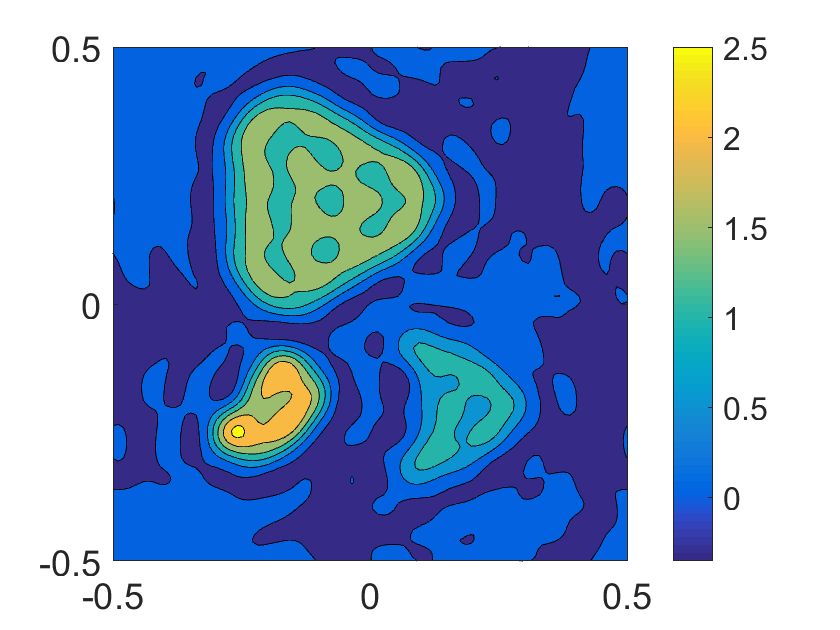}}
   \subfigure[]{\includegraphics[width=0.33\textwidth]{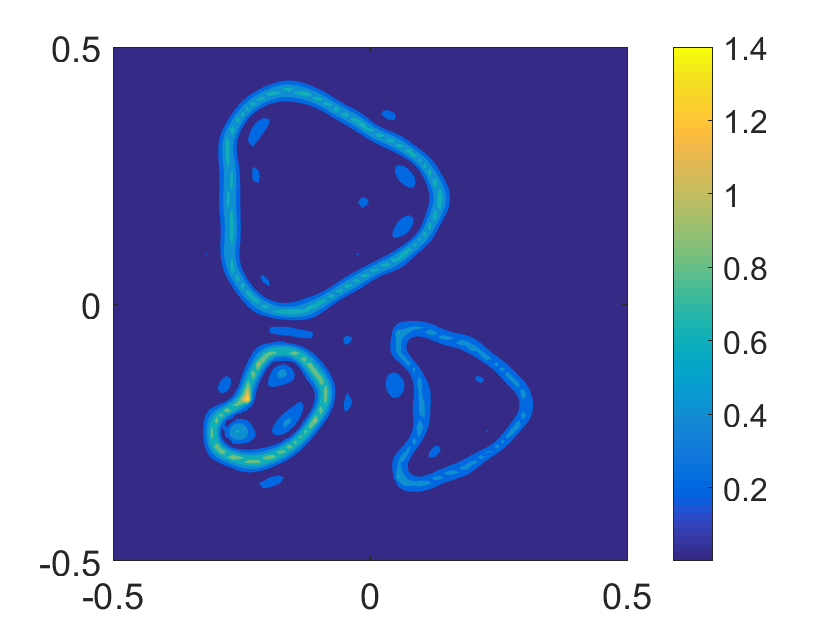}}
   \caption{The exact $S_2$ and the reconstructions $S_2^{\epsilon}$ with noise levels $\epsilon$.
   (a) Exact $S_2$,  (b) $S_2^{\epsilon}$  with $\epsilon=1\%$, (c) $|S_2^{\epsilon}-S_2|$  with $\epsilon=1\%$, (d) $S_2^{\epsilon}$  with $\epsilon=5\%$, (e) $|S_2^{\epsilon}-S_2|$  with $\epsilon=5\%$.}\label{fig:discontinuous}
\end{figure}

\subsection{ Examples in three-dimensions}
\begin{example}\label{E3}
\rm Reconstruction of a  source function in three-dimensions with noise. In this example, we aim to reconstruct an acorn-shaped acoustic source defined in the cube $D=(-0.5,0.5)^3$ by
$$
   S_3(x_1, x_2, x_3)=\mathrm{exp}\left(-\frac{20\sqrt{x_1^2+x_2^2+x_3^2}}{0.6+\sqrt{4.25+2\, \cos3\theta}}\right), \quad \theta=\arccos \frac{x_3}{\sqrt{x_1^2+x_2^2+x_3^2}}.
$$
\end{example}

Figure \ref{fig:hulu-phase} shows the geometry illustration of phase retrieval technique for four typical wavenumbers in the three dimensions. The relative $L^2$ and $L^{\infty}$  errors between the exact- and reconstructed phase data are listed in Table \ref{tab:test2}.
Figure \ref{fig:source-S3} presents the exact source function $S_3$ and the reconstructed results with noise $\epsilon=0.1\% $ via the isosurface plots and slice plots. In the isosurface plots, the gray shadows are projections of the isosurface on each coordinate plane.
These results demonstrate satisfactory imaging performance of the proposed algorithms in three dimensions.
\begin{figure}
\centering
   \subfigure[]{\includegraphics[width=0.48\textwidth]{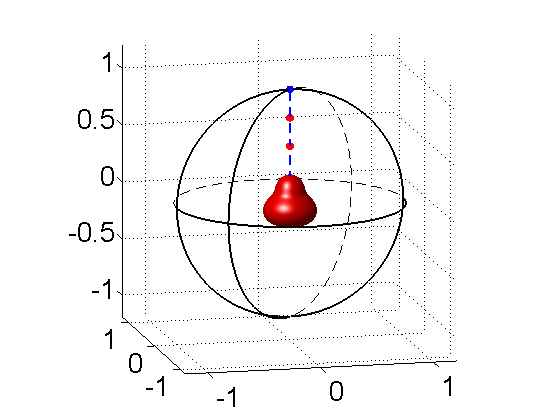}}
   \subfigure[]{\includegraphics[width=0.48\textwidth]{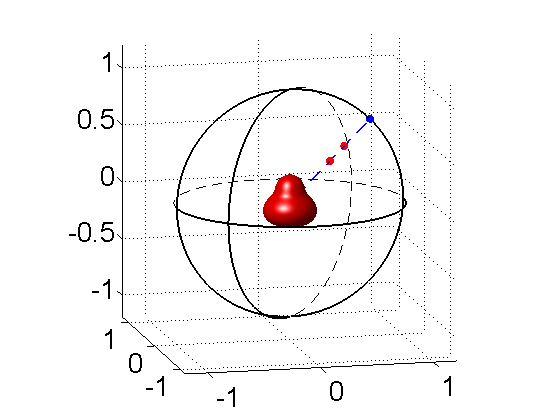}}\\
   \subfigure[]{\includegraphics[width=0.48\textwidth]{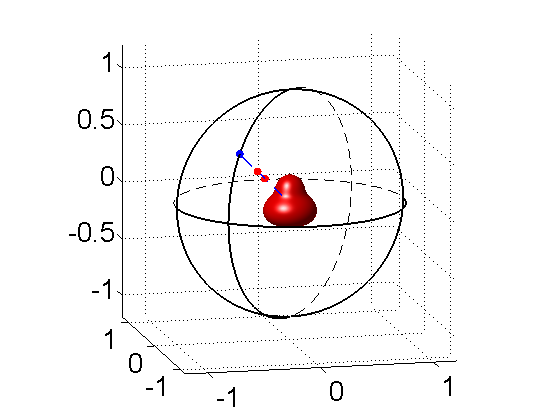}}
   \subfigure[]{\includegraphics[width=0.48\textwidth]{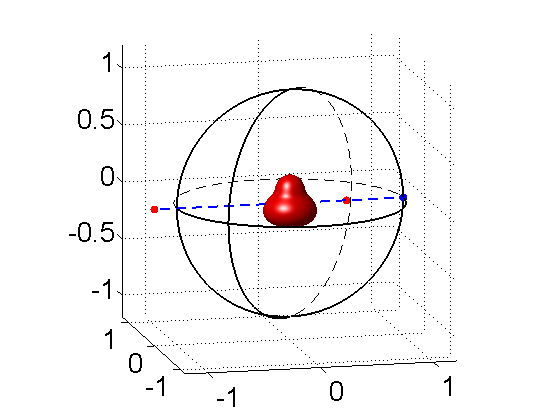}}
   \caption{Geometry illustration of the phase retrieval for Example \ref{E3}. The red points denote the reference point sources and the blue points denote observation directions. (a) $\bm l=(0,1,1)$, (b) $\bm l=(1,0,1)$,  (c) $\bm l=(-1,-1,1)$,  (b) $\bm l=(0,0,0)$. }\label{fig:hulu-phase}
\end{figure}

\begin{figure}
   \centering
   \subfigure[]{\includegraphics[width=0.4\textwidth]{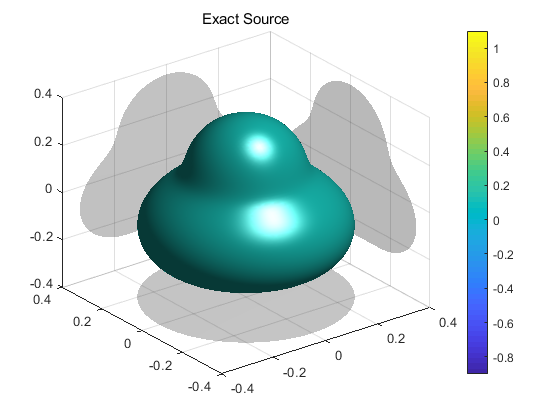}}
   \subfigure[]{\includegraphics[width=0.4\textwidth]{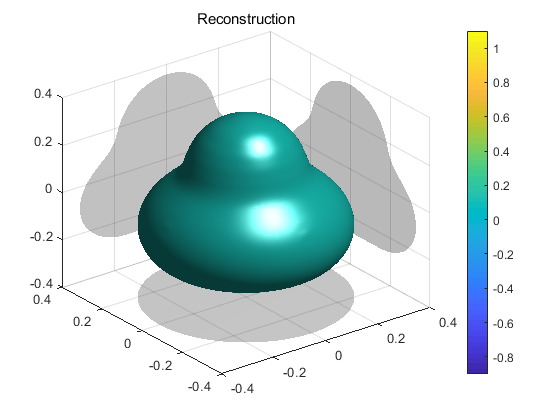}}\\
   \subfigure[]{\includegraphics[width=0.4\textwidth]{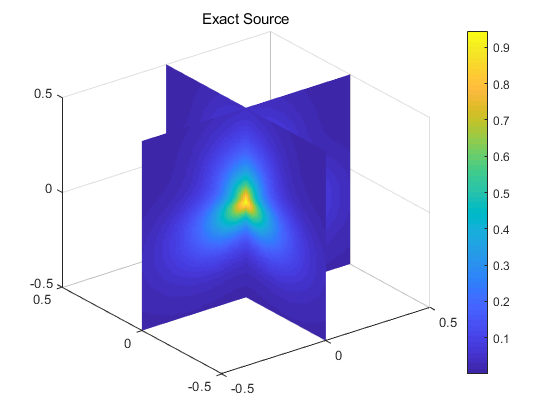}}
   \subfigure[]{\includegraphics[width=0.4\textwidth]{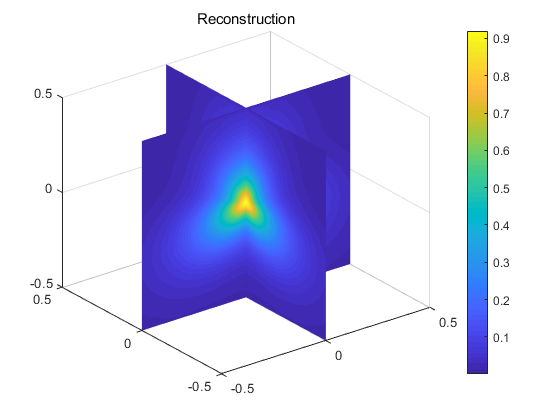}}\\
  \caption{(a) The exact $S_3$ (isosurface level=0.1), (b) reconstruction of  $S_3$ (isosurface level=0.1), (c) the exact $S_3$ (slices at $x_1=0$ and $x_2$=0), (d) reconstruction of   $S_3$ (slices at $x_1=0$ and $x_2$=0)}\label{fig:source-S3}
\end{figure}

\begin{table}[htbp]
 \setlength{\abovecaptionskip}{0pt}
 \setlength{\belowcaptionskip}{10pt}
\centering
 \caption{\label{tab:test2}The relative $L^2$ errors and $L^{\infty}$  errors between $u_\infty (\hat x,k)$ and $u_\infty ^\epsilon(\hat x,k)$ for different noise levels $\epsilon$.}
 \begin{tabular}{lclclcl}
  \toprule
   $S_2(x)$ & $\epsilon=0.1\%$ & $\epsilon=1\%$ & $\epsilon=5\% $ &$\epsilon=10\%$\\
  \midrule
  $L^2$          & $0.06\%$ & $0.56\%$ & $3.10\%$ & $5.97\%$  \\
  $L^{\infty}$  & $0.09\%$ & $0.91\%$ & $3.60\%$ & $10.16\%$ \\
  \bottomrule
 \end{tabular}
\end{table}

\begin{example}\label{E4}
\rm In the last example, we consider the reconstruction of a 3D mountain-shaped source defined in the cube $D=(-0.5,0.5)^3$ by
\begin{align*}
   S_4(x_1, x_2, x_3) = & 1.1\mathrm{exp}\left(-200\left( (x_1-0.01)^2+(x_2-0.12)^2+x_3^2\right)\right) \\ 
   & +100\left(x_1^2-x_2^2\right)\mathrm{exp}\left(-90\left( x_1^2+x_2^2+x_3^2\right)\right).
\end{align*}
\end{example}
Figure \ref{fig:source-S4} illustrates the exact source function $S_4$ and the reconstruction with $0.1\%$ noise. It can be seen that the reconstruction are very close to the exact source $S_4$. To quantitatively exhibit the accuracy, we also list the relative $L^2$ and $L^{\infty}$ errors in Table \ref{tab:test3}.
\begin{figure}
	\centering
	\subfigure[]{\includegraphics[width=0.32\textwidth]{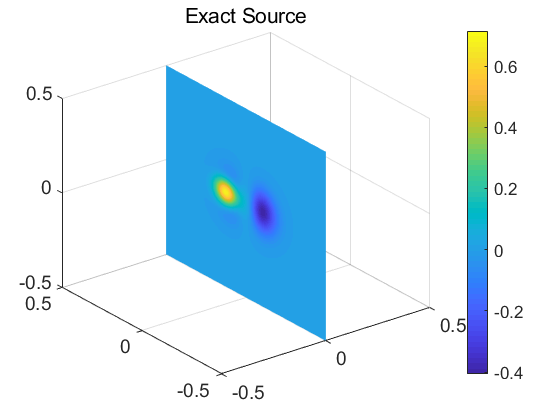}}
	\subfigure[]{\includegraphics[width=0.32\textwidth]{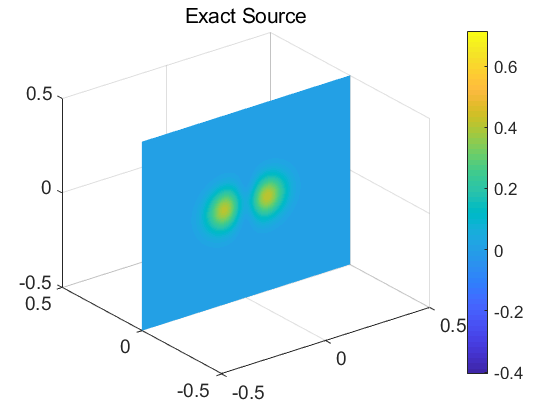}}
        \subfigure[]{\includegraphics[width=0.32\textwidth]{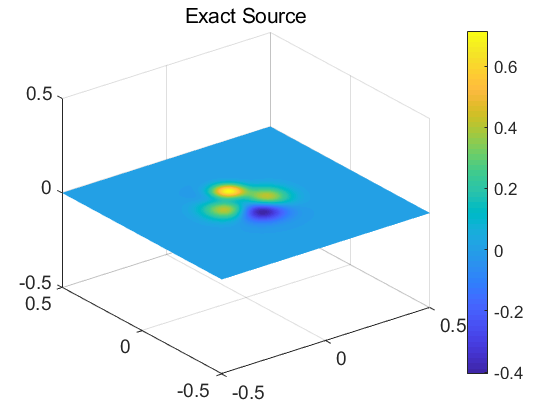}}
	\subfigure[]{\includegraphics[width=0.32\textwidth]{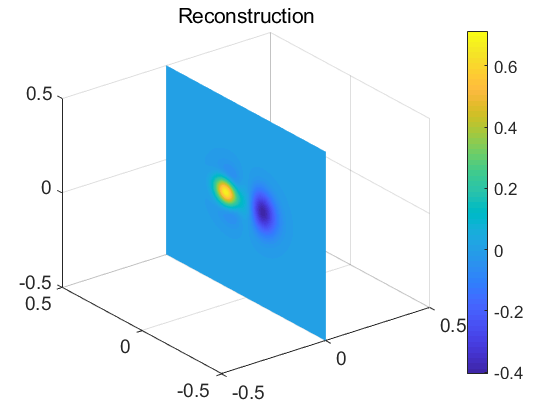}}
        \subfigure[]{\includegraphics[width=0.32\textwidth]{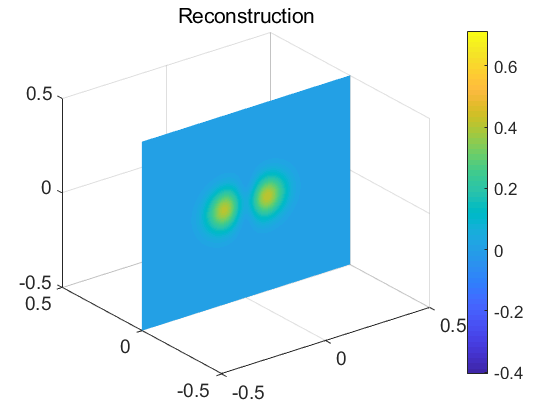}}
	\subfigure[]{\includegraphics[width=0.32\textwidth]{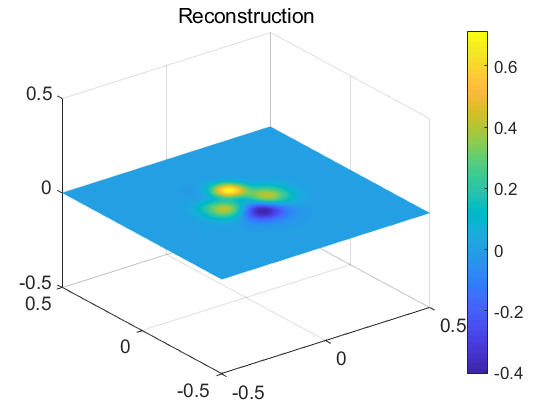}}
	\caption{Top row: slices of the exact source $S_4$. Bottom row: slices of the reconstructed results of $S_4$ . Slice at: (a)(d) $x_1$=0, (b)(e) $x_2$=0, (c)(f) $x_3$=0.}\label{fig:source-S4}
\end{figure}
\begin{table}[htbp]
\setlength{\abovecaptionskip}{0pt}
 \setlength{\belowcaptionskip}{10pt}
\centering
 \caption{\label{tab:test3}The relative $L^2$ errors and $L^{\infty}$  errors between $u_\infty (\hat x, k)$ and $u_\infty^\epsilon(\hat x, k)$ for different noise levels $\epsilon$.}
 \begin{tabular}{lclclcl}
  \toprule
   $S_3(x)$ & $\epsilon=0.1\%$ & $\epsilon=1\%$ & $\epsilon=5\% $ &$\epsilon=10\%$\\
  \midrule
  $L^2$          & $0.06\%$ & $0.55\%$ & $2.85\%$ & $5.89\%$ \\
  $L^{\infty}$  & $0.10\%$ & $0.71\%$ & $3.95\%$ & $9.10\%$ \\
  \bottomrule
 \end{tabular}
\end{table}

\section*{Acknowledgments}

The work of D. Zhang and F. Sun were supported by NSF of China under the grant 11671170. The work of Y. Guo was supported by NSF of China under the grants 11971133, 11601107 and 11671111.



\begin{thebibliography}{10}

\bibitem{AHN19} Agaltsov A D, Hohage T and Novikov R G 2019 An iterative approach to monochromatic phaseless inverse scattering {\it Inverse Problems} {\bf 35} 024001

\bibitem{AM06} {Albanese R and Monk P} 2006 The inverse source problem for Maxwell's equations {\it Inverse Problems} {\bf 22} 1023--1035


\bibitem{AHLS17} {Alzaalig A, Hu G, Liu X and Sun J} 2017 Fast acoustic source imaging using multi-frequency sparse data  {\it arXiv:1712.02654v1 }

\bibitem{ABF02} {Ammari H, Bao G and Fleming J} 2002 An inverse source problem for Maxwell's equations in magnetoencephalography    {\it SIAM J. Appl. Math.} {\bf 62} 1369--1382

\bibitem{AZML07} {Anastasio M A, Zhang J, Modgil D and La Rivi$\grave{\mathrm{e}}$re P J} 2007  Application of inverse source concepts to photoacoustic tomography  {\it Inverse Problems} {\bf 23}  21--35

\bibitem{Arr99} {Arridge S R} 1999  Optical tomography in medical imaging  {\it Inverse Problems} {\bf 15}  R41--R93


\bibitem{BLLT15} {Bao G, Li P, Lin J and Triki F} 2015  Inverse scattering problems with
multi-frequencies {\it Inverse Problems} {\bf 31}  093001

\bibitem{BLT11} {Bao G, Lin J and Triki F} 2011 Numerical solution of the inverse source problem for the Helmholtz equation with multiple frequency data {\it Contemp. Math. AMS}  {\bf 548} 45--60

\bibitem{BLRX15} {Bao G, Lu S, Rundell W and Xu B} 2015   A recursive algorithm for multi-frequency acoustic inverse source problems {\it SIAM J. Numer. Anal.} {\bf 53}  1023--1035

\bibitem{BZ16} Bao G and Zhang L 2016 Shape reconstruction of the multi-scale rough surface from multi- frequency phaseless data {\it Inverse Problems} {\bf 32} 085002

\bibitem{DLU19} {Deng Y, Liu H and Uhlmann G} 2019 On an inverse boundary problem arising in brain imaging {\it J. Differential Equations} {\bf 267} 2471--2502.

\bibitem{DZG19} Dong H, Zhang D and Guo Y 2019 A reference ball based iterative algorithm for imaging acoustic obstacle from phaseless far-field data {\it Inverse Problems and Imaging} {\bf 13} 177--195

\bibitem{FKM04} Fokas A, Kurylev Y and Marinakis V 2004  The unique determination of neuronal currents in the brain via magnetoencephalography {\it Inverse Problems}  {\bf 20}  1067--1082

\bibitem{IK11} {Ivanyshyn O and Kress R} 2011 Inverse scattering for surface impedance from phaseless far field data {\it J. Comput. Phys.} {\bf 230}  3443--3552

\bibitem{JLZ18} Ji X, Liu X and Zhang B 2018 Phaseless inverse source scattering problem: phase retrieval, uniqueness and direct sampling methods {\it arXiv:1808.02385v1}

\bibitem{JLZ19} Ji X, Liu X and Zhang B 2019 Target reconstruction with a reference point scatterer using phaseless far field patterns {\it SIAM J. Imaging Sci.} {\bf 12} 372--391

\bibitem{Kli17} Klibanov M V 2017 A phaseless inverse scattering problem for the 3D Helmholtz equation {\it Inverse Problems Imaging} {\bf 11} 263--76

\bibitem{LU15} {Liu H and Uhlmann G} 2015 Determining both sound speed and internal source in thermo-and photo-acoustic tomography   {\it Inverse Problems} {\bf 31}  105005

\bibitem{SU09} {Stefanov P and Uhlmann G}  2009  Thermoacoustic tomography with variable sound speed {\it Inverse Problems}  {\bf 25}  075011

\bibitem{WGLL17} {Wang X, Guo Y, Li J and Liu H} 2017 Mathematical design of a novel input/instruction device using a moving acoustic emitter  {\it Inverse Problems} {\bf 33} 105009

\bibitem{WGLL19} Wang X, Guo Y, Li J, Liu H 2019 Two gesture-computing approaches by using electromagnetic waves {\it Inverse Problems and Imaging}, {\bf 13} 879--901

\bibitem{WMGL18} {Wang G, Ma F, Guo Y and Li J} 2018 Solving the multi-frequency electromagnetic inverse source problem by the Fourier method   {\it J. Differential Equations}  {\bf 265} 417--443

\bibitem{WGZL17} {Wang X, Guo Y, Zhang D and Liu H} 2017 Fourier method for recovering acoustic sources from multi-frequency far-field data  {\it Inverse Problems} {\bf 33}  035001

\bibitem{WSGLL19} Wang X, Song M, Guo Y, Li H and Liu H 2019 Fourier method for identifying electromagnetic sources with multi-frequency far-field data {\it J. Comput. Appl. Math.} {\bf 358} 279--292

\bibitem{ZG15} {Zhang D and Guo Y} 2015  Fourier method for solving the multi-frequency inverse source problem for the Helmholtz equation  {\it Inverse Problems } {\bf 31}  035007

\bibitem{ZGLL18} {Zhang D,  Guo Y, Li J and Liu H} 2018 Retrieval of acoustic sources from multi-frequency phaseless data  {\it Inverse Problems } {\bf 34}  094001

\bibitem{ZGLL19} Zhang D, Guo Y, Li J and Liu H 2019 Locating multiple multipolar acoustic sources using the direct sampling method {\it Commun. Comput. Phys.} {\bf 25} 1328--1356

\bibitem{ZZ17} {Zhang B and Zhang H} 2017  Recovering scattering obstacles by multi-frequency phaseless far-field data   {\it Journal of Computational Physics}  {\bf 345}  58--37

\end{thebibliography}
\end{document}